\documentclass[10pt]{amsart}
\usepackage{amsthm,amssymb,latexsym,amsmath}

\def \C {\mathbb{C}}

\def \fol {{\mathcal F}}
\def \F {{\mathcal F}}

\def \nn {\mathcal{N}}

\def \P {\mathbb{P}}
\def \pn {\mathbb{P}^n}

\def \sing {{\rm Sing}}
\def \singf {{\rm Sing}(\mathcal{F})}

\newtheorem{maintheorem}{Theorem}

\newtheorem{secondtheorem}{Theorem}

\newtheorem{proposition}{Proposition}[section]

\newtheorem{conjecture}[proposition]{Conjecture}
\newtheorem{corollary}[proposition]{Corollary}

\newtheorem{theorem}[proposition]{Theorem}

\begin{document}

\title[Absolutely $k$-convex domains and holomorphic foliations]{Absolutely $k$-convex domains and holomorphic foliations on homogeneous  manifolds}

\author{Maur\'icio  Corr\^ea Jr. }
\thanks{The first author is partially supported by CNPq grant number 300352/2012-3 and FAPEMIG grant number PPM-00169-13.
 The second author is partially supported by IMPA}
\address{\noindent Maur\' \i cio Corr\^ea Jr\\
Departamento de Matem\'atica \\
Universidade Federal de Minas Gerais\\
Av. Ant\^onio Carlos 6627 \\
30123-970 Belo Horizonte MG, Brazil}
\email{mauricio@mat.ufmg.br}

\author{Arturo Fern\'andez-P\'erez }

\address{Arturo Fern\'andez P\'erez \\
Departamento de Matem\'atica \\
Universidade Federal de Minas Gerais\\
Av. Ant\^onio Carlos 6627 \\
30123-970 Belo Horizonte MG, Brazil} \email{arturofp@mat.ufmg.br}


\begin{abstract}
We consider a holomorphic foliation $\fol$ of codimension $k\geq 1$ on a homogeneous 
compact K\"ahler manifold $X$ of dimension $n>k$. Assuming that the singular set $\sing(\fol)$ of $\fol$ is 
contained in an absolutely $k$-convex domain $U\subset X$,  we prove that the determinant of normal bundle 
$\det(N_{\F})$ of $\fol$ cannot be an ample line bundle, provided $[n/k]\geq 2k+3$. Here $[n/k]$ denotes the largest integer $\leq n/k.$
\end{abstract}
\maketitle
\section{Introduction}

\par Motived by the recently results of J.E. Forn\ae ss, N. Sibony, E.F. Wold \cite{FSE}, we study properties on 
absolutely $k$-convex spaces and holomorphic foliations of arbitrary codimension on homogeneous compact K\"ahler manifolds. The existence of these domains 
in the foliated manifold implies properties of positive for the normal bundle of the foliation involved. 
More precisely, using 
Ohsawa-Takegoshi-Demailly's decomposition \cite{De} for absolutely
$q$-convex spaces
and the residual formulas of Baum-Bott type \cite{baum}, we prove the following
result.
\begin{maintheorem}\label{main_theorem}
Let $\fol$ be a holomorphic foliation, of codimension $k\geq 1$, on a
homogeneous compact K\"ahler manifold $X$ of dimension $n>k$. Suppose
that $Sing(\fol)$  is contained in an absolutely
$k$-convex domain $U\subset X$ and such that $$[n/k]\geq 2k+3.$$
Then, the determinant of normal bundle 
$\det(N_{\F})$ of $\fol$ cannot be an ample line bundle.
\end{maintheorem}

\par Recently smooth foliations on homogeneous compact K\"ahler manifolds have been investigated by Lo Bianco-Pereira \cite{JV}. It follows from Theorem \ref{main_theorem} that a foliation $\F$ of codimension $k\geq 1$ on a homogeneous compact K\"ahler manifold $X$ of dimension $n>k$, with
$\det(N_{\F})$ ample and $[n/k]\geq 2k+3$ is either regular or  the singular set $Sing(\fol)$  cannot be  contained in an absolutely $k$-convex domain $U\subset X$. 
\par On the other hand, related papers about codimension-one 
holomorphic foliations with ample normal bundle on compact K\"ahler
manifolds of dimension at least three have been studied by Brunella in \cite{brunella}, \cite{tori} and
Brunella-Perrone in \cite{perrone}. Furthermore, Brunella stated in \cite{brunella} the following conjecture:
\begin{conjecture}\label{conje}
Let $X$ be a compact connected complex manifold of dimension $n\geq 3$, and let $\fol$
be a codimension-one holomorphic foliation on $X$ whose normal
bundle $N_{\fol}$ is ample. Then every leaf of $\fol$
accumulates to $\singf$.
\end{conjecture}
\par In \cite{perrone}, Brunella-Perrone
proved the Conjecture \ref{conje} for codimension-one holomorphic
foliations on projective manifolds with cyclic Picard group. In the case of $X=\mathbb{P}^n$, $n\geq 3$, the above 
conjecture was proved by Lins Neto \cite{Ln}. We remark that the conjecture \ref{conje} 
can be enunciated in a high codimensional version.
\begin{conjecture}[Generalized Brunella's conjecture]\label{conje1}
Let $X$ be a compact connected complex manifold of dimension $n\geq
3$, and let $\fol$ be a holomorphic foliation of codimension $k<n$ on $X$ whose normal bundle $N_{\fol}$ is ample. Then every
leaf of $\fol$ accumulates to $\singf$, provided $n\geq 2k+1$.
\end{conjecture}
\par Our main result suggests that the property of accumulation of the leaves of $\fol$ to singular set of 
$\fol$ (\textit{or nonexistence of minimal sets of $\fol$} \cite{clp}) depends of the existence of strongly $q$-convex spaces contained singularities of $\fol$ on $X$.  

\par  Suppose that $\fol$ is a codimension-one foliation on $\pn$, $n\geq 3$.
Then its singular set $\singf$ always contains at least one irreducible
component of codimension $2$ (cf. \cite{Ln}). This fact is a
consequence of Baum-Bott formula and turns to be fundamental in the proof nonexistence of non-singular real-analytic 
Levi-flat hypersurfaces due to Lins Neto \cite{Ln}. In order to prove Theorem \ref{main_theorem}, we need prove an analogous result for holomorphic foliations of arbitrary dimension. Of course, we prove the following result, which is valid for foliations with determinant of normal bundle ample on compact complex manifolds.
\begin{secondtheorem}\label{singu}
Let $\mathcal{F}$ be a singular holomorphic foliation of
codimension $k\geq 1$ on a compact complex manifold $X$, such that $\mathrm{cod}
\singf\geq k+1$. If $\det(N_{\fol})$ is ample, then $\singf$ must
have at least one irreducible component of  codimension $k+1$.
\end{secondtheorem}
\par The proof of Theorem \ref{singu} is inspired on Jouanolou's proof in \cite[Proposition 2.7, pg. 97]{J}. Jouanolou supposes that the conormal sheaf
$N_{\mathcal{F}}^*$ of $\fol$ is locally free and ample.
The condition that $N_{\mathcal{F}}^*$ to be  locally free imposes strong restrictions on the singular set of the foliation $\F$, since in this case  $\F$
is given by a locally decomposable holomorphic twisted holomorphic form along to singular set of $\F$.
We will show that these hypotheses are not necessary.

\par This note is organized as follows: in Section 2, we recall some definitions and known results about holomorphic foliations of arbitrary dimension on complex manifolds.
Section 3 is devoted to prove Theorem \ref{singu}. In Section 4, we recall the Baum-Bott formula. In Section 5, we give some definitions and results about $q$-complete spaces and holomorphic foliations.
Finally, in Section 6, we proved Theorem \ref{main_theorem}.
\section{Higher codimensional holomorphic foliations}
Let $X$ be a complex manifold. A \emph{holomorphic foliation} $\F$,
of codimension $k\geq 1$, on $X$   is determined by a nonzero coherent
subsheaf $T\F\subsetneq T_X$, of generic rank $n-k$, satisfying
\begin{enumerate}
    \item[(i)] $\F$ is closed under the Lie bracket, and
    \item[(ii)] $\F$ is saturated in $T_X$ (i.e., $T_X / T\F$ is torsion free).
\end{enumerate}
The locus of points where $T_X / T\F$ is not locally free is called
the singular locus of $\F$, denoted here by $\singf$.

Condition $(i)$ allows us to apply Frobenius Theorem to ensure that
for every point $x$ in the complement of $\singf$, the germ of $T\F$
at $x$ can be identified with the relative tangent bundle of a germ
of smooth fibration $f : (X, x) \rightarrow ( \mathbb{C}^{k},0)$.
Condition $(ii)$  implies that $T\F$ is reflexive and of codimension
of $\singf$ is at least two.

There is a dual point of view where $\mathcal F$ is determined by a
subsheaf $N^*_{ \mathcal F}$, of generic rank $k$, of the cotangent
sheaf $\Omega^1_X = T^* X$ of $X$. The sheaf $N_{\mathcal{F}}^*$ is
called \emph{conormal sheaf} of $\fol$. The involutiveness asked for
in condition $(i)$ above is replace by integrability: if $d$ stands
for the exterior derivative then $dN_{\mathcal{F}}^* \subset
N_{\mathcal{F}}^* \wedge \Omega^1_X$ at the level of local sections.
Condition $(ii)$ is unchanged: $\Omega^1_X /N_{\mathcal{F}}^*$ is
torsion free.

The normal bundle $N_{ \mathcal F}$ of $\mathcal F$ is defined as
the dual of $N_{\mathcal{F}}^* $. We have the following exact
sequence
\[
0 \to T\mathcal F \to TX \to \mathcal{I}_{Z}\cdot N_{\mathcal F} \to 0 \,, 
\]
where $\mathcal{I}_{Z}$ is an ideal sheaf supported in $\sing(\fol)$.
The $k$-th wedge product of the inclusion $N^*_\F\subset \Omega^1_X$
gives rise to a nonzero twisted differential $k$-form $\omega\in
H^0(X,\Omega^{k}_X\otimes \mathcal{N})$ with coefficients in the
line bundle $\mathcal{N}:=\det(N_\F)$, which is \emph{locally
decomposable} and \emph{integrable}. To say that $\omega\in
H^0(X,\Omega^{k}_X\otimes \mathcal{N})$ is locally decomposable
means that, in a neighborhood of a general point of $X$, $\omega$
decomposes as the wedge product of $k$ local $1$-forms
$\omega=\eta_1\wedge\cdots\wedge\eta_{k}$. To say that it is
integrable means that for this local decomposition one has
$$
d\eta_i\wedge \eta_1\wedge\cdots\wedge\eta_{k}=0,   \ \ \forall\ \
i=1,\ldots,k.
$$
Conversely, given a twisted $k$-form $\omega\in
H^0(X,\Omega^{k}_X\otimes \mathcal{N})\setminus\{0\}$ which is
locally decomposable and integrable, we define a foliation of
codimension  $k$ on $X$ as the kernel of the morphism
$$\imath_{\omega}:TX \to \Omega^{k-1}_X\otimes \mathcal{N}$$
given by the contraction with $\omega$.

Let $Y$ be an analytic subset of $X$ pure codimension $k$. We say
that $Y$ is invariant by $\F$ if $\omega_{|Y}\equiv 0$, where
$\omega\in H^0(X,\Omega^{k}_X\otimes \mathcal{N})$ is the twisted
$k$-form inducing $\F$.

\par We specialize to the case $X=\P^n$. In this context, let $\mathcal F$ be a singular holomorphic foliation  on $\mathbb P^n$, of codimension $k\geq 1$,   given by  a
locally decomposable and integrable  twisted $k$-form
$$
\omega \in \mathrm H^0(\mathbb P^n,\Omega^{k}_{\mathbb P^n} \otimes
\mathcal N).
$$

 The degree of $\mathcal F$, denoted by
$\deg(\mathcal F)$, is by definition the degree of the zero locus of
$i^*\omega$, where $i:\mathbb P^{k}\to \mathbb P^n$ is a linear
embedding of a generic $k$-plane. Since $\Omega^{k}_{\mathbb P^{k}}
= \mathcal O_{\mathbb P^q}(-k-1)$ it follows at once that $\mathcal
N = \mathcal O_{\mathbb P^n}(\deg(\mathcal F)+k+ 1)$. In particular,
$\mathcal N$ is ample.

The vector space $  \mathrm H^0(\mathbb P^n,\Omega^{k}_{\mathbb P^n}
\otimes \mathcal O_{\mathbb P^n}(\deg(\mathcal F)+k+ 1))$  can be
canonically
 identified with the vector space of $k$-forms on $\mathbb
C^{n+1}$ with homogeneous coefficients of degree $d+1$ whose
contraction with
 the radial (or Euler) vector field $\mathcal{R}=\sum_{i=0}^n x_i
\frac{\partial}{\partial x_i}$ is identically zero \cite{J}.

\par When $\F$ is a holomorphic foliation on $\mathbb{P}^2$. It is well known that an
algebraic curve $C$ invariant by $\F$  cannot be disjoint to the singularities of $\fol$. In fact, it
follows from Camacho-Sad index Theorem \cite{CS} that
$$
0<\deg(C)^2=\deg(N_C|_{C})=\sum_{p\in \singf \cap C} CS(\F,C,p).
$$
Then $\singf \cap C \neq \emptyset.$ Furthermore, we have the following.

\begin{proposition}
Let $X$ be a  projective manifold and $\F$ a singular holomorphic
foliation, of codimension $k\geq 1$,  on $X$. Let $Y\subset X$ be a closed
subscheme of pure codimension $k$ invariant by $\F$, and $N$ the
normal sheaf of $Y$. Assume $Pic(X)=\mathbb{Z}$, and that there is a
closed curve $C \subset X$, contained in the smooth locus $U$ of $Y$
such that $\deg(N|_{C})>0$. Then $\singf
\cap Y \neq \emptyset$.
\end{proposition}
\begin{proof}
This follows from Esteves-Kleiman's result \cite[
Proposition 3.4, pg. 12]{EK}. In fact, in this case we have that $\singf
\cap Y \neq \emptyset$.
\end{proof}

\section{Proof of Theorem \ref{singu}}
\par Denote by $S=\singf$. Suppose that
$\dim_{\mathbb{C}} S\leq n-k-2$. Consider the cohomological exact
sequence
$$
\cdots \rightarrow H^{2k+1}(M,U,\mathbb{C})\rightarrow
H^{2k+2}(M,\mathbb{C})\stackrel{\zeta}{\rightarrow}
H^{2k+2}(U,\mathbb{C})\rightarrow\cdots
$$
where $U=M\setminus S$. Now consider the Alexander duality
$$
A: H^r(M,U,\mathbb{C})\rightarrow H_{2n-r}(S,\mathbb{C}).
$$
Taking $r=2k+1$ and using that $\dim_{\mathbb{R}}S\leq 2(n-k)-4$, we
conclude that $H_{2(n-k)-1}(S,\mathbb{C})=0.$ In particular,
$H^{2k+1}(M,U,\mathbb{C})=0$ and then the map
$$
H^{2k+2}(M,\mathbb{C})\stackrel{\zeta}{\rightarrow}
H^{2k+2}(U,\mathbb{C})
$$
is injective. On the other hand, by Bott's vanishing Theorem, we
have
$$c^{k+1}_1(N_{\mathcal{F}}|_{U})=0.$$
Since $\zeta(c^{k+1}_1(N_{\mathcal{F}}
))=c^{k+1}_{1}(N_{\mathcal{F}}|_{U})$, we conclude that
$$c^{k+1}_{1}(N_{\mathcal{F}} )=0.$$ This is a contradiction, since
$c_1(N_{\mathcal{F}} )=c_1(\det(N_{\mathcal{F}}))$ and the ampleness
of $\nn=\det(N_{\mathcal{F}})$ implies that the cohomology class
$c^{k+1}_{1}(\det(N_{\mathcal{F}}))$ is non zero.

\section{Baum-Bott formula}

In this section we recall basic facts on Baum-Bott's Theory. For more details see Baum-Bott \cite{baum} and Suwa \cite{suwa}.
\par Let $\fol$ be a holomorphic foliation of codimension $k$ on a complex manifold $X$, $\dim X=n>k$.
Assume that $\fol$ is induced by $\omega\in
H^{0}(X,\Omega^{k}_{X}\otimes\nn)$. Denote by $\sing_{k+1}(\fol)$,
the union of the irreducible components of $\sing (\fol)$ of pure
codimension $k+1$. We are interested in the localization of
\textit{Baum-Bott's class} of $\fol$ over $\sing_{k+1}(\fol)$. Set
\begin{align*}
X^0=X\setminus\sing(\fol)\,\,\,\,\,\,\,\,\,\text{and}\,\,\,\,\,\,\,\,\,\,\,\,X^{*}=X\setminus\sing_{k+1}(\fol).
\end{align*}
\par Take $p_{0}\in X^0$, then in a neighborhood $U_{\alpha}$ of $p_{0}$, $\omega$ decomposes as the wedge product of $k$
local $1$-forms
$\omega_{\alpha}=\eta^{\alpha}_1\wedge\cdots\wedge\eta^{\alpha}_{k}$.
It follows from De Rham Division theorem that the Frobenius
condition
\begin{align}
d\eta^{\alpha}_{\ell}\wedge
\eta^{\alpha}_1\wedge\cdots\wedge\eta^{\alpha}_{k}=0,\,\,\,\,\,\,\forall\,\,\,\ell=1,\ldots,k,
\end{align}
 is equivalent to find a
matrix of holomorphic 1-forms $(\theta^{\alpha}_{\ell s})$, $1\leq
\ell,s\leq k$ satisfying
\begin{align}\label{equa_theta}
d\eta^{\alpha}_{\ell}=\sum^{k}_{s=1}\theta^{\alpha}_{\ell s}\wedge\eta^{\alpha}_{s},\,\,\,\,\,\,\,\forall\,\,\,\ell=1,\ldots,k.
\end{align}
Let $\theta_{\alpha}:=\sum^{k}_{\ell=1}(-1)^{\ell-1}\theta^{\alpha}_{\ell\ell}.$
On $U_{\alpha}\cap U_{\beta}\neq\emptyset$, we have
$\omega_{\alpha}=g_{\alpha\beta}\omega_{\beta},$ where
$g_{\alpha\beta}\in\mathcal{O}^{*}(U_{\alpha}\cap U_{\beta})$ and
$\{g_{\alpha\beta}\}$ defines $\nn$ so that
$d\omega_{\alpha}=dg_{\alpha\beta}\wedge\omega_{\beta}+g_{\alpha\beta}d\omega_{\beta}$.
From (\ref{equa_theta}), we find
$$\left (\frac{dg_{\alpha\beta}}{g_{\alpha\beta}}+\sum^{k}_{\ell=1}(-1)^{\ell-1}\theta^{\beta}_{\ell\ell}-\sum^{k}_{\ell=1}(-1)^{\ell-1}\theta^{\alpha}_{\ell\ell}\right)\wedge\omega_{\alpha}=0,$$
which means that
$$\gamma_{\alpha\beta}:=\frac{dg_{\alpha\beta}}{g_{\alpha\beta}}+\theta_{\beta}-\theta_{\alpha}$$
is a section of $N^{*}_{\fol}$, over $U_{\alpha}\cap U_{\beta}$.
Hence $\{\gamma_{\alpha\beta}\}$ is a cocycle of 1-forms vanishing
on $\fol$, and it corresponds to a cohomology class in
$H^{1}(X,N^{*}_{\fol})$. By taking the cup product $k$-times, we
have the natural map
$$H^{1}(X,N^{*}_{\fol})\otimes\ldots\otimes H^{1}(X,N^{*}_{\fol})\rightarrow H^{k}(X,\nn^{*}),$$
and so we get a class in $H^{k}(X,\nn^{*})$ associated to
$\{\gamma_{\alpha\beta}\}$. This class (in $H^{k}(X,\nn^{*})$) is
intrinsically defined by the foliation, that is, it does not depend
of the choice made so far.
\par On the other hand, in the singular case, the Saito-De Rham Division theorem \cite{rham} implies that the above construction can be made on $X^{*}$. Hence we get a well defined class
(\textit{Baum-Bott's class of $\fol$})
$$BB_{\fol}\in H^{k}(X^{*},\nn^{*})$$ which is intrinsically associated to $\fol$.
\par Let $p\in\sing_{k+1}(\fol)$. We say that $BB_{\fol}$ extends through $p$ if there is a small ball $B_{p}\subset X$ centered at $p$ such that $BB_{\fol}$  extends to a class in $H^{k}(X^{*}\cup B_{p},\nn^{*})$. Denoting
\begin{align*}
S(B_{p})=\sing_{k+1}(\fol)\cap
B_{p}\,\,\,\,\,\,\,\,\,\text{and}\,\,\,\,\,\,\,\,\,\,\,\,B^{*}_{p}=B_{p}\setminus
S(B_{p}),
\end{align*}
and applying Mayer-Vietoris argument, we observe that $BB_{\fol}$ extends through $p$ if and only if $$BB_{\fol}|_{B^{*}_{p}}=0$$
for some ball $B_{p}$  centered at $p$.

\par Now we state Baum-Bott's formula, which is related to the extendibility of the class $BB_{\fol}$ from $X^{*}$ to $X$.
In this sense, we consider
$\omega=\eta_1\wedge\ldots\wedge\eta_{k}$ a local generator of
$\fol$ at $p$ and smooth sections of $N^{*}_{\fol}$ instead of
holomorphic ones, we have the cohomology group $H^{1}(B^{*}_{p},
N^{*}_{\fol})$ is trivial, and so it is possible find a matrix of
smooth $(1,0)$-forms $(\theta_{ls})$, where $\theta_{\ell s}\in
A^{1,0}(B^{*}_{p})$, $1\leq \ell,s\leq k$, such that
\begin{align}
d\eta_{\ell}=\sum^{k}_{s=1}\theta_{\ell s}\wedge\eta_{s},\,\,\,\,\,\,\,\forall\,\,\,\ell=1,\ldots,k.
\end{align}
\par As before, set $\theta=\sum^{k}_{\ell=1}(-1)^{\ell-1}\theta_{\ell\ell}$. Observe that the smooth $(2k+1)$-form
$$\frac{1}{(2\pi i)^{k+1}}\theta\wedge\underbrace{d\theta\wedge\ldots\wedge d\theta}_{k-th}$$ is closed and it has  a
De Rham cohomology class in $H^{2k+1}(B^{*}_{p},\mathbb{C})$ and
moreover it does not depend on the choice of $\omega$ and $\theta$.
\par Let $Z$ be an irreducible component of $\sing_{k+1}(\fol)$. Take a generic point $p\in Z$, that is, $p$ is a point
where $Z$ is smooth and disjoint from the other singular components.
Pick $B_{p}$ a ball centered at $p$ sufficiently small, so that
$S(B_p)$ is a subball of $B_p$ of dimension $n-k-1$. Then the De
Rham class can be integrated over an oriented $(2k+1)$-sphere
$L_{p}\subset B^{*}_{p}$ positively linked with $S(B_{p})$:
$$BB(\fol, Z)=\frac{1}{(2\pi i)^{k+1}}\int_{L_{p}}\theta\wedge (d\theta)^{k}.$$
This complex number is the \textit{Baum-Bott residue of $\fol$ along Z}. It does not depend on the choice of the generic point $p\in Z$.

\par Now we state the main result of this section.
The proof can be found in \cite{baum} or \cite[Theorem VI.3.7]{suwa}
in more general context. We recall that every irreducible component
$Z$ of $\sing_{k+1}(\fol)$ has a fundamental class $[Z]\in
H^{2k+2}(X,\C)$ (conveniently defined via the integration current
over Z).
\begin{theorem}[Baum-Bott \cite{baum}]\label{BB} Let $\F$ be a holomorphic foliation of codimension $k$ on a  complex manifold $X$.   Then the following hold:
\begin{itemize}
\item[(i)]  for each irreducible component $Z$ of $\ \sing_{k+1}(\fol)$ there exist
complex numbers $\lambda_Z(\F)$ which is determined by the local behavior of $\F$ near $Z$.

\item[(ii)] If $X$ is compact,
$$c^{k+1}_{1}(\mathcal{N})=\sum_{Z}\lambda_Z(\F)[Z],$$
where the sum is done over all irreducible components of  $\
\sing_{k+1}(\fol)$.
\end{itemize}
\end{theorem}

Let $U_{0}$ be a neighborhood of $\sing_{k+1}(\fol)$, then we have
that
$$
\sum_{Z}\lambda_Z(\F)[Z]=j^*Res_{c_{1}^{k+1}}(\F,
\sing_{k+1}(\fol)),
$$
where $Res_{c_{1}^{k+1}}(\F, \sing_{k+1}(\fol)) \in
H^{2(n-k)-2}(U_{0},\C)^{*} $ is a cocycle and
$$j*:  H^{2(n-k)-2}(U_{0},\C)^{*}\simeq  H^{ 2k+2}(X,X \backslash \sing_{k+1}(\fol),\C)\to  H^{ 2k-2}(X,\C) $$
is the induced map of the inclusion $j: (X,\emptyset) \to (X,X
\backslash \sing_{k+1}(\fol)  ).$ For more details about it, we refer \cite{brasselet}. 
\par In   \cite{baum} the complex numbers $\lambda_Z(\F)$ are not given  explicitly. We will show that
$$
\lambda_Z(\F)=BB(\fol,Z).
$$
This was proved by Brunella and Perrone in \cite{perrone} when
$k=1$.
\subsection{Proof of Theorem \ref{BB}}
We cover $X$ by open sets $U_{\alpha}$ where the foliation is
defined by holomorphic $k$-forms
$\omega_{\alpha}=\eta^{\alpha}_1\wedge\cdots\wedge\eta^{\alpha}_{k}$
with $\omega_{\alpha}=g_{\alpha\beta}\omega_{\beta}$.  As before, it
is possible find a matrix of $(1,0)$-forms $(\theta_{\ell s})$, where
$\theta_{\ell s}\in A^{1,0}(B^{*}_{p})$, $1\leq \ell,s\leq k$, such that
\begin{align}
d\eta^{\alpha}_{\ell}=\sum^{k}_{s=1}\theta^{\alpha}_{\ell s}\wedge\eta^{\alpha}_{s},\,\,\,\,\,\,\,\forall\,\,\,\ell=1,\ldots,k.
\end{align}
We fix a small neighborhood $V$ of $ \sing_{k+1}(\fol)$ and choose a
matrix of $(1,0)$-forms smooth $(\tilde{\theta}^{\alpha}_{\ell s})$ such
that $\tilde{\theta}^{\alpha}_{\ell s}$ coincide with
$\theta^{\alpha}_{\ell s}$ outside of $U_{\alpha}\cap V$. Let
$\tilde{\theta}_{\alpha}=\sum^{k}_{\ell=1}(-1)^{\ell-1}\tilde{\theta}^{\alpha}_{\ell\ell}$.
Then the smooth $(1,0)$-forms
$$\tilde{\gamma}_{\alpha\beta}=\frac{dg_{\alpha\beta}}{g_{\alpha\beta}}+\tilde{\theta}_{\beta}-\tilde{\theta}_{\alpha}$$
vanish on $\fol$ outside of $V$. This cocycle can be trivialized:
$\tilde{\gamma}_{\alpha\beta}=\tilde{\gamma}_{\alpha}-\tilde{\gamma}_{\beta}$,
where $\tilde{\gamma}_{\alpha}$ is a smooth $(1,0)$-form on
$U_{\alpha}$ vanishing on $\fol$ outside of $U_{\alpha}\cap V$.
Therefore, after setting
$\hat{\theta}_{\alpha}=\tilde{\theta}_{\alpha}+\tilde{\gamma}_{\alpha}$,
we find
$$\frac{dg_{\alpha\beta}}{g_{\alpha\beta}}=\hat{\theta}_{\alpha}-\hat{\theta}_{\beta}.$$

Hence, $\Theta=\frac{1}{2\pi i}d\hat{\theta}_{\alpha}$ is a globally defined closed 2-form which represents, in the De Rham sense, the Chern class of $\det(N_{\fol})=\nn$. Therefore,
$$\Theta^{k+1}:=\frac{1}{(2\pi i)^{k+1}}\underbrace{d\hat{\theta}_{\alpha}\wedge\ldots\wedge d\hat{\theta}_{\alpha}}_{(k+1)-th}$$ represents $c^{k+1}_{1}(N_{\fol})$.
It follows from Bott's vanishing theorem that $\Theta^{k+1}=0$
outside $V$, that is,
$$Supp(\Theta^{k+1})\subset \overline{V}.$$
\par If $T\subset X$ is a $(k+1)$-ball intersecting transversally $\sing_{k+1}(\fol)$ at a single point $p\in Z$, with $V\cap T\Subset T$, then by Stokes formula
\begin{eqnarray*}
BB(\fol, Z) & = & \frac{1}{(2\pi i)^{k+1}}\int_{\partial T}\hat{\theta}_{\alpha}\wedge (d\hat{\theta}_{\alpha})^{k} \\
& = & \frac{1}{(2\pi i)^{k+1}}\int_{T}(d\hat{\theta}_{\alpha})^{k+1}
\end{eqnarray*}
This means that the $2(k+1)$-form $\Theta^{k+1}$ is cohomologous, as
a current, to the integration current over $$ \sum_{Z}BB(\F,Z)[Z].
$$

\section{Strongly $q$-convex spaces}
\par In this section, we present some results about \textit{strongly $q$-convex spaces}. These results will be applied in the study of invariant sets of holomorphic foliations on complex manifolds. The concept of $q$-convexity was first introduced by Rothstein \cite{ro} and further developed by Andreotti-Grauert \cite{AG}. More details about it can be found in Demailly's book \cite{dema}.
\par Let $(M,\mathcal{O}_M)$ be a complex analytic space, possibly non reduced. Recall that a function $\varphi:M\to\mathbb{R}$ is said to be  strongly $q$-convex in the sense of Andreotti-Grauert \cite{AG} if there exists a covering of $M$ by open patches $A_{\lambda}$ isomorphic to closed analytic sets in open sets $U_{\lambda}\subset \C^{N_{\lambda}}$, $\lambda\in I$, such that each restriction $\varphi|_{A_{\lambda}}$ admits an extension $\tilde{\varphi}_{\lambda}$ on $U_{\lambda}$ which is strongly $q$-convex, $i.e.$ such that $i\partial\bar{\partial}\tilde{\varphi}_{\lambda}$ has at most $q-1$ negative or zero eigenvalues at each point of $U_{\lambda}$. Note that the strong $q$-convexity property does not depend on the covering nor on the embeddings $A_{\lambda}\subset U_{\lambda}$.    
\par The space $M$ is said to be \textit{strongly  $q$-complete, resp. strongly $q$-convex}, if $M$ has a smooth exhaustion function $\varphi$ such that $\varphi$ is strongly $q$-convex on $M$, resp. on the complement $M\setminus K$ of a compact set $K\subset M$. From \cite{De}, $M$ is said to be \textit{absolutely  $q$-convex} if it admits a smooth plurisubharmonic exhaustion function $\varphi:M\setminus K\to \mathbb{R}$ that is strongly $q$-convex on $M\setminus K$ for some compact set $K\subset M$.

\par We will use Ohsawa-Takegoshi-Demailly's Theorem \cite{Oh} and  Andreotti-Grauert vanishing theorem \cite{AG}.
\begin{theorem}[Ohsawa-Takegoshi-Demailly]\label{Ohsawa}
Let $U$ be an absolutely  $q$-convex K\"ahler manifold of dimension $n$. Then the De Rham cohomology groups with arbitrary supports have   decomposition
$$H^{k}(U,\mathbb{C}) \simeq \bigoplus_{s+\ell=k} H^{s}(U, \Omega^{\ell}), \ \ H^{s}(U, \Omega^{\ell})\simeq \overline{H^{\ell}(U, \Omega^{s}) }\     \ k\geq n+q.$$

\end{theorem}

\begin{theorem}[Andreotti-Grauert]\label{vanishing}
Let $U$ be a $q$-complete manifold of dimension $n$. For any coherent holomorphic sheaf $\mathcal{G}$ on $U$ and any
$j\geq q$, we have $$H^{j}(U, \mathcal{G})=0.$$
\end{theorem}

\par We recall that a function $\varphi:M\to\mathbb{R}$ is strongly $q$-convex with corners on $M$ if for every point $p\in M$ there is a neighborhood $U_p$ and finitely many strongly $q$-convex functions $\{\varphi_{p,j}\}_{j\leq \ell_p}$ on $U_p$ such that $\varphi|_{U_p}=\max_{j\leq \ell_p}\{\varphi_{p,j}\}$. The manifold $M$ is said to be \textit{strongly $q$-convex with corners} if it admits an exhaustion function which is strongly $q$-convex with corners outside a compact set. Similarly, $M$ is said to be \textit{strongly $q$-complete with corners} if it admits an exhaustion function which is strongly $q$-convex with corners.
\par In \cite{DF}, Diederich and Forn\ae ss proved the following result.
\begin{theorem}[Diederich-Forn\ae ss]\label{D-F}
Any $q$-convex ($q$-complete) manifold $U$ with corners, $\dim U=n$,
is $\tilde{q}$-convex ($\tilde{q}$-complete)
 with $\tilde{q}=n-\left[\frac{n}{q}\right]+1$.
\end{theorem}

We will also use the following result by M. Peternell for
homogeneous manifolds \cite{P}.

\begin{theorem}[Peternell]\label{P}
If $X $ is a homogeneous compact complex manifold and
$U\varsubsetneq X$ is a open set in $X$ that is $q$-convex with
corners then $U$ is $q$-complete with corners.
\end{theorem}

It follows from Theorem \ref{D-F} and Theorem \ref{P} the following.
\begin{corollary}\label{Homoge}
If $X $ is a homogeneous compact complex manifold and
$U\varsubsetneq X$ is a open set in $X$ that is  $q$-convex with
corners then $U$ is $\tilde{q}$-complete
 with $\tilde{q}=n-\left[\frac{n}{q}\right]+1$.
\end{corollary}

\par To prove Theorem \ref{main_theorem}, we need prove the following.
\begin{theorem}\label{domain-corner}
Let $\fol$ be a holomorphic foliation,  of codimension $k\geq 1$, on a
homogeneous compact K\"ahler manifold $X$ of dimension $n$. Suppose
that $Sing_{k+1}(\fol)$  is contained in an absolutely
$k$-convex  open  $U\subset X$ and that  $$[n/k]\geq 2k+3.$$
Then, $j^*Res_{c_{1}^{k+1}}(\F, \sing_{k+1}(\fol))=0$.
\end{theorem}
\begin{proof}
First of all, it follows from Corollary \ref{Homoge} that $U$ is
$\tilde{k}$-complete with $$\tilde{k}=n-\left[\frac{n}{k}\right]+1.$$
Since  $U\subset X$ is absolutely  $k$-convex  and
$$2(n-k)-2\geq n+n-[n/k] +1 \geq n+k$$
it follows from Ohsawa-Takegoshi-Demailly's Theorem that
$$
H^{2(n-k)-2}(U,\mathbb{C}) \simeq \bigoplus_{s+\ell=2(n-k)-2}
H^{s}(U, \Omega^{\ell}), \ \ H^{s}(U, \Omega^{\ell})\simeq
\overline{H^{\ell}(U, \Omega^{s}) }.
$$
On the other hand,  the condition $2(n-k)-2\geq 2n-[n/k] +1$ implies
that is either $s\geq n-[n/k] +1$ or $\ell\geq n-[n/k] +1$. In fact,
suppose that $s< n-[n/k] +1$ and  $\ell < n-[n/k] +1$. Then
$$2(n-k)-2=s+\ell <  2n-2[n/k] +2<2n-[n/k] +1,$$ absurd since
$2(n-k)-2\geq 2n-[n/k] +1$.

Now, if  $s\geq \tilde{k}=n-[n/k] +1$ we have
$$
 H^{s}(U, \Omega^{\ell})=0
$$
by Andreotti-Grauert's vanishing Theorem, since $U$ is
$\tilde{k}$-complete. Otherwise, if $s< n-[n/k] +1$,  then
$\ell\geq n-[n/k] +1$ and by Andreotti-Grauert's vanishing Theorem $
H^{\ell}(U, \Omega^{s})=0$. but, by 
 Ohsawa-Takegoshi-Demailly's Theorem we have
$$
H^{s}(U, \Omega^{\ell})\simeq \overline{H^{\ell}(U, \Omega^{s}) }=0.
$$
Therefore, $H^{2(n-k)-2}(U,\mathbb{C}) =0$. That is,
$H^{2(n-k)-2}(U,\mathbb{C})^* =0$. In particular
$j^*Res_{c_{1}^{n-q+1}}(\F, \sing_{q-1}(\fol))=0.$
\end{proof}

\section{Proof of Theorem \ref{main_theorem}}
Let  $\mathcal F$ be  a singular holomorphic foliation of
codimension $k\geq 1$ and suppose by contradiction that $\nn=\det(N_{\fol})$ is ample.
Now, by Baum-Bott formula (Theorem \ref{BB}), we
have
\begin{equation}\label{formula}
c^{k+1}_{1}(\mathcal{N})=\sum_{Z}BB(\fol,Z)[Z]=j^*Res_{c_{1}^{k+1}}(\F,
\sing_{k+1}(\fol)),
\end{equation}
where the sum is done over all irreducible components of
$\sing_{k+1}(\fol)$.  Because $\nn$ is ample, the class
$c^{k+1}_{1}(\mathcal{N})$ is not zero, and by Theorem
\ref{singu}, we infer that $\sing(\fol)$ always has irreducible
components of codimension $k+1$ and so this sum is not zero.
But the hypotheses over $\singf$ implies that $\sing_{k+1}(\fol)\subset U$, where $U$ is an absolutely
$k$-convex domain. 
Applying Theorem \ref{domain-corner}, we must have
$j^*Res_{c_{1}^{k+1}}(\F, \sing_{k+1}(\fol))=0$. It is a
contradiction with (\ref{formula}).
\vspace{1cm}

\noindent {\bf Acknowledgements.}  We would like to thank Marcio G.
Soares for his comments and suggestions.

\end{document}